\documentclass[12pt]{article}

\usepackage{verbatim}
\usepackage{graphics}
\usepackage[pdftex]{graphicx}
\usepackage{amsmath, amssymb, amsfonts, amsthm}
\usepackage{url}
\usepackage{setspace}
\usepackage{algorithmic}
\usepackage{enumitem}
\setlist{nolistsep}

\theoremstyle{plain}
\newtheorem{theorem}{Theorem}
\newtheorem{lemma}[theorem]{Lemma}

\theoremstyle{definition}

\usepackage[colorlinks=true,citecolor=black,linkcolor=black,urlcolor=blue]{hyperref}

\newcommand{\arxiv}[1]{\href{http://arxiv.org/abs/#1}{\texttt{arXiv:#1}}}

\begin{document}

\title{An improved bound on $(A+A)/(A+A)$.}
\author{Ben Lund \thanks{Supported by NSF grant CCF-1350572.}
\date{}
}

\maketitle

\begin{abstract}
We show that, for a finite set $A$ of real numbers, the size of the set
$$\frac{A+A}{A+A} = \left\{ \frac{a+b}{c+d} : a,b,c,d \in A, c+d \neq 0 \right \}$$
is bounded from below by
$$\left|\frac{A+A}{A+A} \right| \gg \frac{|A|^{2+1/4}}{|A / A|^{1/8} \log |A|}.$$
This improves a result of Roche-Newton (2016).

\end{abstract}

\section{Introduction}

Given a finite set $A$ of real numbers, we define its \textit{sum set} to be
$$A+A = \{a+b : a,b \in A\},$$
its \textit{product set} to be
$$AA = \{ab : a, b \in A\},$$
and its \textit{quotient set} to be
$$A/A = \{a/b : a,b \in A\}.$$

Erd\H{o}s and Szemer\'edi \cite{erdos1983sums} conjectured that, for any $\epsilon > 0$,\footnote{
The notation $f(n) \gg g(n)$ indicates that $f(n)$ is bounded below by a constant times $g(n)$.}
$$\max(|AA|, |A+A|) \gg |A|^{2- \epsilon}.$$

In this paper we study the related question of establishing lower bounds on the size of the set
$$\frac{A+A}{A+A} = \left\{ \frac{a+b}{c+d} : a,b,c,d \in A, c+d \neq 0 \right \}.$$

In \cite{balog2015new}, Balog and Roche-Newton showed that, if $A$ is a set of strictly positive reals,
$$\left | \frac{A+A}{A+A} \right | \geq 2|A|^2 - 1.$$
This is completely sharp, as shown by the set $A = \{1,2,3\}$.

In \cite{roche2014bound}, Roche-Newton and Zhelezov conjectured that
\begin{equation}\label{eq:smallA+A}\left | \frac{A+A}{A+A} \right | \ll |A|^2 \Rightarrow |A+A| \ll |A|.\end{equation}
Shkredov \cite{shkredov2016difference} has made some progress in this direction, showing that
\begin{equation}\left |\frac{A-A}{A-A} \right| \ll |A|^2 \Rightarrow |A-A| \ll |A|^{2-1/5} \log ^{3/10}|A|.
\end{equation}
In \cite{roche2016if}, Roche-Newton proved that
\begin{equation}\label{eq:OliversResult}\left | \frac{A+A}{A+A} \right | \gg \frac{|A|^{2+2/25}}{|A/A|^{1/25} \log |A|},\end{equation}
which implies that
\begin{equation}\label{eq:largeAA} \left | \frac{A+A}{A+A} \right | \ll |A|^2 \Rightarrow |A/A| \gg \frac{|A|^2}{ \log^{25}|A|}.\end{equation}
It is  known that (\ref{eq:largeAA}) is implied by (\ref{eq:smallA+A}), for example by work of Li and Shen \cite{li2010sum}, but the reverse implication does not hold (even if the extra $\log^{-25}|A|$ term were removed).

This paper gives the following improvement to (\ref{eq:OliversResult}) and (\ref{eq:largeAA}).
\begin{theorem}\label{th:main}
Let $A$ be a finite set of real numbers. Then
\begin{equation}\label{eq:main}\left|\frac{A+A}{A+A} \right| \gg \frac{|A|^{2+1/4}}{|A / A|^{1/8} \log |A|}.\end{equation}
Consequently,
\begin{equation}\label{eq:consequence}\left | \frac{A+A}{A+A} \right | \ll |A|^2 \Rightarrow |A/A| \gg \frac{|A|^2}{ \log^{8}|A|}.\end{equation}
\end{theorem}

In broad outline, the proof of Theorem \ref{th:main} is similar to the proof by Roche-Newton of (\ref{eq:OliversResult}).
We combine ideas of Solymosi \cite{solymosi2009bounding} and Konyagin and Shkredov \cite{konyagin2015sum}, developed in work on the Erd\H{o}s-Szemer\'edi sum-product conjecture, with an incidence bound and a probabilitistc argument.
The key difference between this proof and that of Roche-Newton is that we apply the probabilistic method in a more flexible way, which leads to a simpler proof of a stronger result.

\subsection{Acknowledgments}
I would like to thank Abdul Basit, Oliver Roche-Newton, and Adam Sheffer for useful conversations related to the material in this paper.

\section{Proof of Theorem \ref{th:main}}

The remainder of the paper is dedicated to the proof of inequality (\ref{eq:main}) Theorem \ref{th:main}; given (\ref{eq:main}), it is trivial to derive (\ref{eq:consequence}).

In section \ref{sec:setUp} we describe the general setup for the proof and fix notation.
In section \ref{sec:overview}, we give the general idea of the proof, and briefly elaborate on the key difference between this proof and that of Roche-Newton \cite{roche2016if}.
Sections \ref{sec:slopes}, \ref{sec:representatives}, and \ref{sec:combining} comprise the main body of the proof.

\subsection{Setup}\label{sec:setUp}
We assume that all of the elements of $A$ are strictly positive.
This is without loss of generality; if at least half of the elements of $A$ are positive, we consider these elements; otherwise, we multiply by $-1$ and then consider the positive elements.

Using a dyadic pigeonholing argument, we find a set $P \subset A \times A$ and a number $ |A|^2/\left(2|A/A|\right) \leq \tau \leq |A|$ such that $P$ is contained in the union of lines through the origin (in $\mathbb{R}^2$) that each contain exactly $\tau$ points, and
\begin{equation}\label{eq:sparsify}|P| \gg |A|^2 / \log |A|.\end{equation}
In more detail, $A \times A$ is contained in the union of $|A/A|$ lines through the origin, one for each ratio in $A/A$.
Note that no more than $|A|^2/2$ points in $A \times A$ are contained in lines through the origin that each contain fewer than $\tau_0 = |A|^2/\left(2|A/A|\right)$ points.
Note also that no line contains more than $|A|$ points of $A \times A$.
We partition the numbers $[\tau_0,|A|]$ into $O(\log |A|)$ dyadic intervals of the form $[2^i,2^{i+1})$, and use the pigeonhole principle to show that there is a $\tau$ such that there are at least $\Omega(|A|^2 / \log |A|)$ points $(a,b) \in A \times A$ that lie on lines through the origin that each contain between $\tau$ and $2\tau$ points.
We then choose $\tau$ arbitrary points on each of these lines, and define the resulting point set to be $P$.

Note that, if $a,b,c,d$ are elements of $A$, then the point $(a+b, c+d) \in (A \times A) + (A \times A)$ is contained in the line through the origin with slope $(c+d)/(a+b)$.
Hence, $(A \times A) + (A \times A)$ is contained in the union of $(A+A)/(A+A)$ lines through the origin.
Since $P+P$ is a subset of $(A \times A) + (A \times A)$, it suffices to show that, for the set $S$ of lines through the origin that each contain at least one point of $P+P$, we have
\begin{equation}\label{eq:goal}|S| \geq \frac{|A|^{2+1/4}}{|A/A|^{1/8} \log|A|}.\end{equation}
Once (\ref{eq:goal}) is demonstrated, the proof will be complete.

Let $\Lambda = \{\lambda_1, \lambda_2, \ldots, \lambda_{|\Lambda|}\}$ with $\lambda_1 < \lambda_2 < \ldots \lambda_{|\Lambda|}$ be the set of slopes of lines through the origin that each contain $\tau$ points of $P$.
Note that
\begin{equation}\label{eq:PinLines}|P| = \tau |\Lambda|.\end{equation}

Let $M$ be an integer parameter that we will set later.
For each $1 \leq i \leq \lfloor |\Lambda| / 2M \rfloor$, let
\begin{align*}
f_i &= 2M(i-1), \\
T_i &= \{\lambda_{f_i+1}, \lambda_{f_i + 2}, \ldots, \lambda_{f_i + M}\}, \\
U_i &= \{\lambda_{f_i+M+1}, \lambda_{f_i+M+2}, \ldots, \lambda_{f_i + 2M}\}.
\end{align*}

For the remainder of the proof, we work with an arbitrary $i$ and set $T = T_i$ and $U = U_i$.
We relabel $\lambda_{f_i + j}$ as $\lambda_{j}$, so that $\lambda_1, \lambda_2, \ldots, \lambda_M \in T$ and $\lambda_{M+1}, \lambda_{M+2}, \ldots, \lambda_{2M} \in U$.
With this relabeling, let $P_i$ be the set of points of $P$ contained in the line with slope $\lambda_i$.

\subsection{Brief overview}\label{sec:overview}
The basic idea of the proof is to show that we can select points $a_{ij} \in P_i$ for $1 \leq i \leq M$ and $M+1 \leq j \leq 2M$ so that the union of the sets $a_{ij}+P_j$ determines many slopes.
There are two basic parts of the proof.
First, we use a proof based on a geometric incidence bound to show that, for fixed $i,j,k,\ell$ with $(i,j) \neq (k,\ell)$, there are many selections of $a_{ij}, a_{k\ell}$ such that $a_{ij}+P_j$ shares few slopes with $a_{k\ell} + P_\ell$.
Second, we combine this with a probabilistic argument to show that we can select the $a_{ij}$ so that the total number of slopes is large.

The key difference between this proof and the proof of Roche-Newton in \cite{roche2016if} is that we choose a point $a_{ij}$ for each pair in $T \times U$, where Roche-Newton chose a single representative for each line in $T$.

\subsection{Bounding shared slopes}\label{sec:slopes}
For a point $p \in \mathbb{R}^2$, denote by $r(p)$ the slope of the line passing through $p$ and the origin.
For $a_{i} \in P_i$ and $a_{k} \in P_k$, let
\begin{equation}\mathcal{E}(a_{i}, P_j, a_{k}, P_\ell) = |\{(b_j, b_\ell) \in P_j \times P_\ell : r(a_{i} + b_j) = r(a_k + b_\ell)\}|.\end{equation}

The proof of the following lemma is based closely on the work of Roche-Newton \cite{roche2016if}.
\begin{lemma}\label{th:sharedSlopeBound}
Let $1 \leq i,k \leq M$ and $M+1 \leq j,\ell \leq 2M$, with at least one of $i \neq k$ and $j \neq \ell$.
The number of pairs $(a_i, a_k) \in P_i \times P_k$ such that
$$\mathcal{E}(a_{i}, P_j, a_{k}, P_\ell) \geq K$$
is bounded from above by
$$O(\tau^4/K^3 + \tau^2/K).$$
\end{lemma}

\begin{proof}

For each pair of points $(a, \lambda_i a), (b, \lambda_k b)$ in $P_i \times P_k$, we define the curve
$$l_{ab} = \{(x,y) : (\lambda_i a + \lambda_j x)(b + y) = (\lambda_k b + \lambda_\ell y)(a + x)\}.$$

For any set $Q$ of points, denote by $\pi_x(Q)$ the projection of $Q$ onto the $x$-axis; in other words,
$$\pi_x(Q) = \{a : \exists y((a, y) \in Q)\}.$$

Let \begin{align*}
\mathcal{L} &= \{l_{ab}:\left((a,\lambda_i a),(b,\lambda_k b)\right) \in P_i \times P_k\}, \\
\mathcal{P} &= \pi_x(P_j) \times \pi_x(P_\ell).
\end{align*}

Note that $(x,y) \in l_{ab}$ is equivalent to $r((a + x, \lambda_i a + \lambda_j x)) = r((b+y, \lambda_k b + \lambda_\ell y))$.
Hence, to prove the lemma, it suffices to show that there are no more than $\tau^4/K^3 + \tau^2/K$ curves from $\mathcal{L}$ that each contain at least $K$ points of $\mathcal{P}$.
To this end, we use the following result of Pach and Sharir \cite{pach1998number}.
\begin{theorem}\label{th:pachSharir}
Let $\mathcal{L}$ be a family of curves and let $\mathcal{P}$ be a set of points in the plane such that
\begin{enumerate}
\item any two distinct curves from $\mathcal{L}$ intersect in at most two points of $\mathcal{P}$,
\item for any two distinct points $p,q \in \mathcal{P}$, there exist at most two curves $\mathcal{L}$ which pass through both $p$ and $q$.
Then, for any $k \geq 2$, the set $\mathcal{L}_k \subset \mathcal{L}$ of curves that contain at least $k$ points of $\mathcal{P}$ satisfies the bound
$$|\mathcal{L}_k| \ll |\mathcal{P}|^2 k^{-3} + |\mathcal{P}|k^{-1}.$$
\end{enumerate}
\end{theorem}

Since $|\mathcal{P}|=\tau^2$, Lemma \ref{th:sharedSlopeBound} will follow directly from this theorem.
It remains to show that $\mathcal{L}$ satisfies the hypotheses of Theorem \ref{th:pachSharir}.

We will first check that two distinct curves of $\mathcal{L}$ intersect in at most two points.

Let $l_{ab}$ and $l_{a'b'}$ be two distinct curves in $\mathcal{L}$.
Their intersection is the set of all $(x,y)$ such that
\begin{align}
\label{eq:lab} (\lambda_i a + \lambda_j x)(b + y) &= (\lambda_k b + \lambda_\ell y)(a + x), \\
\label{eq:la'b'} (\lambda_i a' + \lambda_j x)(b' + y) &= (\lambda_k b' + \lambda_\ell y)(a' + x).
\end{align}

Let 
\begin{equation}\label{eq:deltas}
\delta_1 = \lambda_j-\lambda_\ell, \qquad
\delta_2 = \lambda_j - \lambda_k, \qquad
\delta_3 = \lambda_\ell- \lambda_i, \qquad
\delta_4 = \lambda_k - \lambda_i.
\end{equation}
Note that $\delta_2, \delta_3 \neq 0$, and at least one of $\delta_1 \neq 0$ or $\delta_4 \neq 0$ holds.
In addition, some simple algebra combined with the observation that $(\lambda_j - \lambda_i)(\lambda_\ell - \lambda_k) \neq 0$ shows that
\begin{equation}\label{eq:d14-23}
\delta_1 \delta_4 - \delta_2 \delta_3 \neq 0.
\end{equation}

Rearrange (\ref{eq:lab}) and (\ref{eq:la'b'}) into the form
\begin{align}
\label{eq:rlab} x(\delta_1 y+ \delta_2 b) &= a(\delta_3 y + \delta_4 b), \\
\label{eq:rla'b'} x(\delta_1 y + \delta_2 b') &= a'(\delta_3 y + \delta_4 b'). 
\end{align}

Suppose, for contradiction, that $\delta_1 y + \delta_2 b = 0$.
This implies that $\delta_1 \neq 0$, since $\delta_2, b \neq 0$.
Then we have $y = -\delta_2 \delta_1^{-1} b$, and from the right side of (\ref{eq:rlab}), we have $\delta_3 y + \delta_4 b = 0$.
From this, we conclude that $-\delta_3 \delta_2 \delta_1^{-1} + \delta_4 = 0$, which contradicts (\ref{eq:d14-23}).
Hence, $\delta_1 y + \delta_2 b \neq 0$, and by a similar argument, we conclude that $\delta_1 y + \delta_2 b' \neq 0$.

Hence, we conclude
\begin{equation}\label{eq:xcombined} a(\delta_3 y + \delta_4 b) (\delta_1 y+ \delta_2 b)^{-1} = x = a'(\delta_3 y + \delta_4 b')(\delta_1 y + \delta_2 b')^{-1}. \end{equation}

From equation (\ref{eq:xcombined}), we get the quadratic equation
\begin{equation}\label{eq:quadratic}\delta_1 \delta_3 (a - a') y^2 + \left(\delta_2\delta_3(ab' - a'b) + \delta_1 \delta_4 (ab - a'b')\right)y + \delta_2\delta_4bb'(a' - a) = 0.\end{equation}
Either there are at most two values of $y$ which give a solution to this quadratic, or all of the coefficients are zero.

Suppose that we are in this degenerate case.
If $\delta_1 \neq 0$, then the fact that the coefficient of $y^2$ is zero implies that $a = a'$.
Otherwise, $\delta_4 \neq 0$, and so the fact that the constant term is zero implies that $a = a'$.
Combining the fact that $a' = a$ with the fact that the linear term is zero, we conclude that $\delta_2 \delta_3 - \delta_1 \delta_4 = 0$, which contradicts (\ref{eq:d14-23}).

Hence, $l_{ab}$ and $l_{a'b'}$ intersect in at most two points.
Now, we show that there are at most two curves of $\mathcal{L}$ that pass through any pair of points in $\mathcal{P}$.

If $l_{xy} \in \mathcal{L}$ is a curve that that passes through $(a,b)$ and $(a',b')$, then
\begin{align}
\label{eq:labNew} (\lambda_i x + \lambda_j a)(y + b) &= (\lambda_k y + \lambda_\ell b)(x + a), \\
\label{eq:la'b'New} (\lambda_i x + \lambda_j a')(y + b') &= (\lambda_k y + \lambda_\ell b')(x + a').
\end{align}

Rearrange these equations into the form
\begin{align}
\label{eq:rlabNew} x(\delta_4 y+ \delta_3 b) &= a(\delta_2 y + \delta_1 b), \\
\label{eq:rla'b'New} x(\delta_4 y + \delta_3 b') &= a'(\delta_2 y + \delta_1 b'). 
\end{align}

These equations are the same as equations (\ref{eq:rlab}) and (\ref{eq:rla'b'}) with $\delta_4$ interchanged with $\delta_1$, and $\delta_3$ interchanged with $\delta_2$.
Since the conditions on these quantities are symmetric with regard to the pairs $\delta_1, \delta_4$ and $\delta_2, \delta_3$, we apply the previous argument to show that there are at most two solutions to this system of equations.
\end{proof}

\subsection{Choosing representatives}\label{sec:representatives}
For each pair $(i,j)$ with $1 \leq i \leq M$ and $M+1 \leq j \leq 2M$, choose $a_{ij} \in P_i$ uniformly at random.
Note that we have chosen $M^2$ points.
For each pair $a_{ij}, a_{k\ell}$ of chosen points, let $X(i,j,k,\ell)$ be the event that
$$\mathcal{E}(a_{ij}, P_j, a_{k\ell}, P_k) > B,$$
where $B$ is a parameter that we will fix later.
Applying Lemma \ref{th:sharedSlopeBound}, we find that, for each quadruple $(i,j,k,\ell)$ with at least one of $i \neq j$ and $k \neq \ell$, we have
\begin{equation}\label{eq:prB}Pr[X(i,j,k,\ell)] \ll \tau^2B^{-3} + B^{-1}.\end{equation}

In addition, note that $X(i,j,k,\ell)$ depends on $X(i',j',k',\ell')$ only if either $(i,j) = (i',j')$ or $(k,\ell) = (k',\ell')$.
Hence, $X(i,j,k,\ell)$ is independent of all but at most $2M^2$ other events.

We apply the following version of the Lov\'asz Local Lemma to bound the probability that at least one of the events occurs - see Corollary 5.1.2 in \cite{alon2015probabilistic}.

\begin{theorem}\label{th:LLL}
Let $X_1, X_2, \ldots, X_n$ be events in an arbitrary probability space. Suppose that each event $X_i$ is mutually independent from all but at most $d$ of the events $X_j$, with $i \neq j$. Suppose also that the probability of the event $X_i$ occuring is at most $p$, for all $1 \leq i \leq n$. Finally suppose that
$$ep(d+1) \leq 1.$$
Then, with positive probability, none of the events $X_1, X_2, \ldots, X_n$ occur.
\end{theorem}

Using (\ref{eq:prB}), we apply this lemma with $p = c_1(\tau^2B^{-3} + B^{-1})$ and $d = 2M^2$, and conclude that, as long as
\begin{equation}\label{eq:constraint}ec_1(\tau^2B^{-3} + B^{-1})(2M^2 + 1) \leq 1,\end{equation}
there is a positive probability that none of the events $X(i,j,k,\ell)$ occur.

\subsection{Combining the bounds}\label{sec:combining}
Let $Q$ be the union of $a_{ij} + P_j$ over all the $a_{ij}$.
Note that $Q$ is the union of $M^2$ distinct sets, each containing $\tau$ points.
Let $R$ be the slopes of lines through points of $Q$.
By inclusion-exclusion,
$$|R| \geq M^2\tau - \sum \mathcal{E}(a_{ij}, P_j, a_{k\ell}, P_k),$$
where the sum is over all $1 \leq i,k \leq M$ and $M+1 \leq j,\ell \leq 2M$ such that at least one of $i\neq k$ and $j\neq \ell$ holds.
Since each $\mathcal{E}(a_{ij},P_j, a_{k\ell}, P_\ell) < B$, we have
\begin{equation}\label{eq:R}|R| \geq M^2 \tau - BM^4.\end{equation}

We now set $B = \tau/(2M^2)$.
Returning to the constraint (\ref{eq:constraint}), we require that
$$ec_1(8M^6 \tau^{-1} + 2M^2\tau^{-1})(2M^2 + 1) \leq 1.$$
It is possible to satisfy this constraint with $M= c_2 \tau^{1/8}$, for an appropriate choice of $c_2$.

Note that the sets $R \subset S$  obtained for different choices of $T,U$ are disjoint.
Since there are $\lfloor |\Lambda| / 2M \rfloor$ choices for $T$ and $U$, we have by equation \ref{eq:R} that the number of lines through the origin that contain some point of $P+P$ is at least
$$|S| = |R| \left \lfloor |\Lambda| / 2M \right \rfloor \gg M \tau |\Lambda| \gg \tau^{1+1/8} |\Lambda|.$$
Applying equation (\ref{eq:PinLines}), we have
$$|S| \gg |P|\tau^{1/8}.$$

Combining this with equation (\ref{eq:sparsify}) and the observation that $\tau \geq |A|^2 / |A/A|$, we have inequality (\ref{eq:goal}), and the proof is complete.

\end{document}